\documentclass[english]{amsart}
\usepackage{amsfonts,amsmath,amssymb,color,txfonts,amsthm,graphics,mathrsfs}

\usepackage[T1]{fontenc}
\usepackage[english]{babel}
\usepackage{color}
\usepackage[T1]{fontenc}
\usepackage[all]{xy}

\usepackage{enumerate}
\usepackage{amscd}
\usepackage{exscale}
\usepackage{latexsym}
\usepackage[all]{xy}
\usepackage{hyperref}
\usepackage{fancyhdr}
\usepackage{layout}
\usepackage{graphicx}

\newtheorem{theorem}{Theorem}[section]
\newtheorem{lemma}{Lemma}[theorem]
\newtheorem{proposition}{Proposition}[theorem]
\newtheorem{corollary}{Corollary}[theorem]

\theoremstyle{definition}
\newtheorem{definition}{Definition}[theorem]

\theoremstyle{remark}

\renewenvironment{proof}{Proof}{\hfill $\square$ \\ \indent}

\newcommand\N{{\mathbb N}}

\newcommand\Q{{\mathbb Q}}

\newcommand\Res{{\mathrm{Res}}}

\newcommand\Z{{\mathbb Z}}

\newcommand\F{{\mathbb F}}

\newcommand\p{{\mathbb P}}

\newcommand\tr{\hbox to 1mm  {${}^t \!  $} }

\newcommand{\nc}{\newcommand}
\nc{\ace}{\`e }
\nc{\aca}{\`a }
\nc{\aci}{\`i }
\nc{\aco}{\`o }
\nc{\acu}{\`u }
\nc{\pid}{\mathfrak{p} }
\nc{\bdm}{\begin{displaymath}}
\nc{\edm}{\end{displaymath}}
\nc{\beq}{\begin{equation}}
\nc{\eeq}{\end{equation}}
\nc{\dpid}{\delta_{\mathfrak{p}}}
\nc{\vvs}{\textquotedblleft}
\nc{\vvd}{\textquotedblright}
\nc{\os}{\mathcal{O_S}}
\nc{\srsu}{R_S^*}

\nc{\mcu}{\mathcal{U}}

\nc{\srs}{R_S}

\title[Preperiodic points over $\F_p(t)$]{On preperiodic points of rational functions\\ defined over $\F_p(t)$}

\author{Jung Kyu Canci}
\address{Jung Kyu Canci, Universit\"{a}t Basel, Mathematisches Institut, Rheinsprung $21$, CH-$4051$ Basel, Switzerland}
\email{jungkyu.canci@unibas.ch}

\author{Laura Paladino}
\thanks{L. Paladino is partially supported by Istituto Nazionale di Alta Matematica, grant research \emph{Assegno di ricerca Ing. G. Schirillo}, and partially supported by the European Commission and by Calabria Region through the
European Social Fund.}
\address{Laura Paladino, Universit\aca di Pisa, Dipartimento di Matematica, Largo Bruno Pontecorvo 5, 56127 Pisa, Italy.}
\email{paladino@mail.dm.unipi.it}

\begin{document}

\begin{abstract}
 Let $P\in\mathbb{P}_1(\mathbb{Q})$ be a periodic point for a monic polynomial with coefficients in $\mathbb{Z}$.  With elementary techniques one sees that the minimal periodicity of $P$ is at most $2$.   Recently  we proved a generalization of this fact to the set of all rational functions defined over $\Q$ with good reduction everywhere (i.e. at any finite place of $\mathbb{Q}$). The set of monic polynomials with coefficients in $\mathbb{Z}$ can be characterized, up to conjugation by elements in  PGL$_2(\Z)$,  as the set of all rational functions defined over $\mathbb{Q}$ with a totally ramified fixed point in $\mathbb{Q}$ and with good reduction everywhere.    Let $p$ be a prime number and let $\F_p$ be the field with $p$ elements. In the present paper we consider rational functions defined over  the  rational global function  field $\F_p(t)$  with good reduction at every finite place. We prove some bounds for the cardinality of orbits in $\F_p(t)\cup \{\infty\}$ for periodic and preperiodic points..
\end{abstract}
\maketitle


\bigskip

\section{Introduction}

In arithmetic dynamic there is a great interest about periodic and preperiodic points of a rational function $\phi\colon \p_1\to \p_1$. A point $P$ is said to be \emph{periodic} for $\phi$ if there exists an integer $n>0$ such that $\phi^n(P)=P$. The minimal number $n$ with the above properties is called \emph{minimal} or \emph{primitive period}. We say that $P$ is a \emph{preperiodic point} for $\phi$ if its (forward) orbit $O_\phi(P)=\{\phi^n(P)\mid n\in\N\}$ contains a periodic point. In other words $P$ is preperiodic if its orbit $O_\phi(P)$ is finite. In this context an orbit is also called a cycle and its size is called the length of the cycle.

  Let $p$ be a prime and, as usual, let $\F_p$ be the field with $p$ elements. We denote by $K$  a global field, i. e. a finite extension of the field of rational numbers $\mathbb{Q}$ or a finite extension of the field $\F_p(t)$. Let $D$ be the degree of $K$ over the base field (respectively $\mathbb{Q}$ in characteristic 0 and $\F_p(t)$ in positive characteristic). We denote by ${\rm PrePer}(\phi,K)$ the set of $K$--rational preperiodic points for $\phi$. By considering the notion of height, one sees that the set ${\rm PrePer}(\phi,K)$ is finite for any rational map $\phi\colon\p_1\to\p_1$ defined over $K$ (see for example \cite{Z} or \cite{HS}). The finiteness of the set ${\rm PrePer}(f,K)$ follows from 
  \cite[Theorem B.2.5, p.179]{HS} and \cite[Theorem B.2.3, p.177]{HS} (even if these last theorems are formulated in the case of number fields, they have a similar statement in the function field case). Anyway, the bound deduced by those results depends strictly on the coefficients of the map $\phi$ (see also \cite[Exercise 3.26 p.99]{Z}).
So, during the last twenty years, many dynamists have searched for bounds that do not depend on the coefficients of $\phi$. In 1994 Morton and Silverman
 stated a conjecture known with the name \vvd
 Uniform Boundedness Conjecture for Dynamical Systems\vvd:  for any number field $K$, the number of $K$-preperiodic points of a morphism $\phi\colon \p_N\to\p_N$ of degree $d \geq 2$, defined over $K$, is bounded by a number depending only on the integers $d,N$ and $D=[K:\Q]$. This conjecture is really interesting even for possible
 application on torsion points of abelian varieties. In fact, by considering the Latt\`es map associated to the multiplication by two map $[2]$ over an elliptic curve $E$, one sees that the Uniform Boundedness Conjeture for $N=1$ and $d=4$ implies  Merel's Theorem on torsion points of elliptic curves (see \cite{Me}).   The Latt\`{e}s map has degree 4 and its preperiodic points are in one-to-one correspondence with the torsion points of $E/\{\pm 1\}$ (see \cite{Sil.2}).
 So a proof of the conjecture for every $N$, could provide an analogous of Merel's Theorem for all abelian varieties. Anyway, it seems very hard to solve this conjecture, even for $N=1$.

 Let $R$ be the ring of algebraic integers of $K$. Roughly speaking: we say that an endomorphism $\phi$ of $\p_1$ has (simple) good reduction at a $place$ $\pid$ if $\phi$ can be written in the form $\phi([x:y]) = [F(x,y),G(x,y)]$, where $F(x,y)$ and $G(x,y)$ are homogeneous polynomial of the same degree with coefficients in the local ring $R_\pid$ at $\pid$ and such that their resultant
$\Res(F,G)$  is a $\pid$--unit. In Section \ref{gr} we present more carefully the notion of good reduction.

The first author studied some problems linked to Uniform Boundedness Conjecture. In particular, when $N=1$, $K$ is a number field and $\phi:\p_1\rightarrow \p_1$ is an endomorphism defined over $K$, he proved in \cite[Theorem 1]{C} that the lenght of a cycle of a preperiodic point of $\phi$  is bounded by a number depending only on the cardinality of the set of places of bad reduction of $\phi$.

 A similar result in the function field case was recently proved in \cite{CP}. Furthermore in the same paper there is a bound proved for number fields, that is slightly better than the one in \cite{C}.

\begin{theorem}[Theorem 1, \cite{CP}] \label{preper} Let $K$ be a global field. Let $S$ be a finite set of places of $K$, containing all the archimedean ones, with cardinality $|S|\geq 1$. Let $p$ be the characteristic of $K$. Let $D=[K:\mathbb{\F}_p(t)]$ when $p>0$, or $D=[K:\mathbb{Q}]$ when $p=0$.  Then there exists a number $\eta(p,D,|S|)$, depending only on $p$, $D$ and  $|S|$, such that if  $P\in\p_1(K)$ is a preperiodic point for an endomorphism   $\phi$ of $\p_1$ defined over $K$   with good reduction outside $S$,   then $|O_\phi(P)|\leq  \eta(p,D,|S|)$.
We can choose
$$\eta(0,D,|S|)=\max\left\{(2^{16|S|-8}+3)\left[12|S|\log(5|S|)\right]^{D}, \left[12(|S|+2)\log(5|S|+5)\right]^{4D}\right\}$$
 in zero characteristic  and
\begin{equation}\label{eta}\eta(p,D,|S|)=(p|S|)^{4D}\max\left\{\left(p|S|\right)^{2D},p^{4|S|-2}\right\}.\end{equation}
 in positive characteristic.
\end{theorem}

\noindent Observe that the bounds in Theorem \ref{preper} do not depend on the degree $d$ of $\phi$. As a consequence of that result, we could give the following bound for the cardinality of the set of $K$-rational preperiodic points for an endomorphism $\phi$ of $\mathbb{P}_1$ defined over $K$.

\begin{corollary}[Corollary 1.1, \cite{CP}] \label{UBCforCGR}Let $K$ be a global field. Let $S$ be a finite set of places of $K$ of cardinality $|S|$ containing all the archimedean ones. Let $p$ be the characteristic of $K$. Let $D$ be the degree of $K$ over the rational function field $\F_p(t)$, in the positive characteristic, and over $\Q$, in the zero characteristic. Let $d\geq 2$ be an integer. Then there exists a number $C=C(p,D,d,|S|)$, depending only on $p$, $D$, $d$ and  $|S|$, such that for any endomorphism $\phi$ of $\p_1$ of degree $d$, defined over $K$ and with good reduction outside $S$, we have
$$\#{\rm PrePer}(\phi, \p_1(K))\leq C(p,D,d,|S|).$$
\end{corollary}

Theorem \ref{preper} extends to global fields and to preperiodic points the result proved by Morton and Silverman in \cite[Corollary B]{MS1}.
The condition $|S|\geq 1$ in its statement is only a technical one. In the case of number fields, we require that $S$ contains the archimedean places (i.e. the ones at infinity), then it is clear that the cardinality of $S$ is not zero. In the function field case any place is non archimedean.  Recall that the place at infinity in the case $K=\F_p(t)$ is the one associated to the valuation given by the prime element $1/t$. When $K$ is a finite extension of $\F_p(t)$, the places at infinity of $K$ are the ones that extend the place of $\F_p(t)$ associated to $1/t$.  The arguments used in the proof of Theorem \ref{preper} and Corollary \ref{UBCforCGR} work also when $S$ does not contain all the places at infinity. Anyway, the most important situation is when all the ones at infinity are in $S$. For example, in order to have that any polynomial in $\F_p(t)$ is an $S$--integer, we have to put in $S$ all those places. Note that in the number field case the quantity $|S|$ depends also on the degree $D$ of the extension $K$ of $\Q$, because $S$ contains all archimedean places (whose amount depends on $D$).

Even when the cardinality of $S$ is small, the bounds in Theorem \ref{preper} is quite big. This is a consequence of our searching for
some uniform bounds (depending only on $p,D,|S|$). The bound $C(p,D,d,|S|)$ in Corollary \ref{UBCforCGR} can be effectively given, but in this case  too the bound is big, even for small values of the parameters $p,D,d,|S|$. For a much smaller bound see for instance the one proved by Benedetto in \cite{B.1} for the case where $\phi$ is a polynomial.  In the   more general case when  $\phi$ is a rational function with good reduction outside a finite $S$,  the bound in Theorem \ref{preper} can be significantly improved for some particular sets $S$.  For example   if $K=\Q$ and $S$  contains
 only the place at infinity,  then  we have the following bounds (see \cite{CP}):

\begin{itemize}
\item If  $P\in \p_1({\mathbb{Q}})$ is a periodic point for $\phi$
with minimal period $n$, then $n\leq 3$.
\item  If $P\in \p_1({\mathbb{Q}})$ is a preperiodic point for $\phi$, then $|O_\phi(P)|\leq 12$.
\end{itemize}

\noindent Here we prove some analogous bounds when $K=\F_p(t)$.

\begin{theorem}\label{cffzs} Let $\phi\colon \p_1\to\p_1$  of degree $d\geq 2$ defined over $\F_p(t)$ with good reduction at every finite place. If $P\in \p_1(\F_p(t))$ is a periodic point for $\phi$ with minimal period $n$, then
\begin{itemize}
\item $n\leq 3$ \hspace{0.1cm} if $p=2$;
\item  $n\leq 72$ \hspace{0.1cm} if $p=3$
\item $n\leq (p^2-1)p$ \hspace{0.1cm} if $p\geq 5$.
\end{itemize}

More generally if $P\in \p_1(\F_p(t))$ is a preperiodic point for $\phi$ we have

\begin{itemize}
\item $|O_\phi(P)|\leq 9$ \hspace{0.1cm} if $p=2$;
\item  $|O_\phi(P)|\leq 288$ \hspace{0.1cm} if $p=3$;
\item  $|O_\phi(P)|\leq (p+1)(p^2-1)p$  \hspace{0.1cm} if $p\geq 5$.
\end{itemize}

\end{theorem}

\noindent Observe that the bounds do not depend on the degree of $\phi$ but they depend only on the characteristic $p$.
 In the proof we will use some ideas already written in \cite{C.1}, \cite{C} and \cite{CP}. The original idea of using $S$--unit theorems in the context of the arithmetic of dynamical systems is due to Narkiewicz \cite{N.1}.

\section{Valuations, $S$-integers and $S$-units}\label{prel}

We adopt the present notation:
let $K$ be a global field and $v_\pid$ a valuation on $K$ associated to a non archimedean place $\pid$. Let $R_\pid=\{x\in K\mid v_\pid (x)\geq 1\}$ be the local ring of $K$ at $\pid$.

Recall that we can associate an absolute value to any valuation $v_p$, or more precisely a place $\pid$ that is a class of absolute values  (see \cite{HS} and \cite{Sti}  for a reference about this topic). With $K=\F_p(t)$, all places are exactly the ones associated either to a monic irreducible polynomial in $\F_p[t]$ or to the place at infinity given by the valuation $v_\infty(f(x)/g(x)=\deg(g(x))-\deg(f(x))$, that is the valuation associated to $1/t$.

In an arbitrary finite extension $K$ of $\F_p(t)$ the valuations of  $K$ are the ones that extend the valuations of $\F_p(t)$. We shall call places at infinity the ones that extend the above valuation $v_\infty$ on $\F_p(t)$. The other ones will be called finite places. The situation is similar to the one in number fields. The non archimedean places in $\Q$ are the ones associated to the valuations at any prime $p$ of $\Z$. But there is also a place that is not non--archimedean, the one associated to the usual absolute value on $\Q$. With an arbitrary number field $K$ we call archimedean places all the ones that extend to $K$ the place given by the absolute value on $\Q$.
\par From now on $S$ will be a finite fixed set of places of $K$. We shall denote by
\bdm R_S \coloneqq\{x\in K \mid v_{\mathfrak{p}}(x)\geq0 \ \text{for every prime }\ \mathfrak{p}\notin S\}\edm
the ring of $S$-integers and by
\bdm R_S^\ast \coloneqq\{x\in K^\ast\mid v_{\mathfrak{p}}(x)=0 \ \text{for every prime }\ \mathfrak{p}\notin S\}\edm
the group of $S$-units.

 As usual let $\overline{\F}_p$ be the algebraic closure of $\F_p$. The case when $S=\emptyset$ is trivial because if so, then the ring of $S$--integers is already finite; more precisely $R_S=R_S^*=K^*\cap \overline{\F}_p$. Therefore in what follows we consider $S\neq \emptyset$.

In any case we have that $K^*\cap \overline{\F}_p$ is contained in $R_S^\ast$.  Recall that the group $R_S^\ast/K^*\cap \overline{\F}_p$ has finite rank equal to $|S|-1$ (e.g. see \cite[Proposition 14.2 p.243]{Ros}). Thus, since $K\cap \overline{\F}_p$ is a finite field, we have that $R_S^*$ has rank equal to $|S|$.

\section{Good reduction} \label{gr}

We shall state the notion of good reduction following the presentation given in \cite{Sil.2} and in \cite{CP}.

\begin{definition}
Let $\Phi:\p_1\to\p_1$ be a rational map defined over $K$, of the
form
$$\Phi([X:Y])=[F(X,Y):G(X,Y)]$$
where $F,G\in K[X,Y]$ are coprime homogeneous polynomials of the same degree. We say that $\Phi$ is in $\pid$--\emph{reduced form} if the coefficients of $F$ and $G$ are in $R_\pid[X,Y]$ and at least one of them is a
$\pid$-unit (i.e. a unit in $R_\pid$).
\end{definition}

Recall that $R_p$ is a principal local ring. Hence, up to multiplying  the polynomials $F$ and $G$ by a suitable non-zero element of $K$, we can always find a $\pid$--reduced form for each rational map.
We may now give the following definition.
\begin{definition}
Let $\Phi:\p_1\to\p_1$ be a rational map defined over $K$.
Suppose that the morphism $\Phi([X:Y])=[F(X,Y):G(X,Y)]$ is written in $\pid$-reduced form.
The \emph{reduced map} $\Phi_\pid:\p_{1,k(\pid)}\to\p_{1,k(\pid)}$ is defined by $[F_\pid(X,Y):G_\pid(X,Y)]$, where $F_\pid$ and $G_\pid$ are the polynomials obtained from $F$ and $G$  by reducing their coefficients modulo $\pid$.
\end{definition}
With the above definitions we give the following one:
\begin{definition}
A rational map $\Phi\colon\p_1\to\p_1$, defined over $K$, has \emph{good reduction}  at $\pid$ if  $\deg\Phi=\deg{\Phi}_\pid$. Otherwise we say that it has bad reduction at $\pid$. Given a set $S$ of places of  $K$ containing all the archimedean ones. We say that $\Phi$ has good reduction outside $S$ if it has good reduction at any place $\pid\notin S$.
\end{definition}

Note that the above definition of good reduction is equivalent to ask that the homogeneous resultant of the polynomial $F$ and $G$ is invertible in $R_\pid$, where we are assuming that $\Phi([X:Y])=[F(X,Y):G(X,Y)]$ is written in $\pid$-reduced form.

\section{Divisibility arguments} \label{dia}

We define the $\mathfrak{p}$-adic logarithmic distance as follows   (see also \cite{MS}).
The definition is independent of the choice of the homogeneous coordinates.

\begin{definition}
Let $P_1=\left[x_1:y_1\right],P_2=\left[x_2:y_2\right]$ be two distinct points in $\mathbb{P}_1(K)$.  We denote by  \begin{equation} \label{d_p}\delta_{\mathfrak{p}}\,(P_1,P_2)=v_{\mathfrak{p}}\,(x_1y_2-x_2y_1)-\min\{v_{\mathfrak{p}}(x_1),v_{\mathfrak{p}}(y_1)\}-\min\{v_{\mathfrak{p}}(x_2),v_{\mathfrak{p}}(y_2)\}\end{equation}the $\mathfrak{p}$-adic logarithmic distance.
\end{definition}

The divisibility arguments, that we shall use to produce the $S$--unit equation useful to prove our bounds, are obtained starting from the following two facts:

\begin{proposition}\label{5.1}\emph{\cite[Proposition 5.1]{MS}}
\begin{displaymath} \delta_{\mathfrak{p}}(P_1,P_3)\geq \min\{\delta_{\mathfrak{p}}(P_1,P_2),\delta_{\mathfrak{p}}(P_2,P_3)\}\end{displaymath}
for all $P_1,P_2,P_3\in\mathbb{P}_1(K)$.\end{proposition}

\begin{proposition}\label{5.2}\emph{\cite[Proposition 5.2]{MS}}
Let $\phi\colon \p_1\to\p_1$ be a morphism defined over $K$ with good reduction at a place $\pid$. Then for any $P,Q\in\p(K)$ we have
\bdm  \delta_{\mathfrak{p}}(\phi(P),\phi(Q))\geq \delta_{\mathfrak{p}}(P,Q).\edm
\end{proposition}

As a direct application of the previous propositions we have the following one.

\begin{proposition}\label{6.1}\emph{\cite[Proposition 6.1]{MS}}
Let $\phi\colon \p_1\to\p_1$ be a morphism defined over $K$ with good reduction at a place $\pid$. Let $P\in\p(K)$ be a periodic point for $\phi$ with
minimal period n. Then
\begin{itemize}
\item $\dpid(\phi^i(P),\phi^j(P))=\dpid(\phi^{i+k}(P),\phi^{j+k}(P))$\ \ for every $i,j,k\in\Z$.
\item Let $i,j\in\N$ such that $\gcd(i-j,n)=1$. Then $\dpid(\phi^i(P),\phi^j(P))=\dpid(\phi(P),P)$.
\end{itemize}
\end{proposition}

\section{Proof of Theorem \ref{cffzs}}

\noindent We first recall the following statements.

 \begin{theorem}[Morton and Silverman \cite{MS}, Zieve \cite{Zie}]\label{mst} Let $K,\pid,p$ be as above. Let $\Phi$ be an endomorphism of $\p_1$ of degree at least two defined over $K$ with good reduction at $\pid$.
Let $P\in \p_1(K)$ be a periodic point for $\Phi$ with minimal period $n$. Let $m$ be the primitive period of the reduction of $P$ modulo $\pid$ and $r$ the multiplicative period of $(\Phi^m)'(P) $ in $k(\pid)^*$. Then one of the following three conditions holds

\begin{itemize}
  \item[(i)] $n=m$;
  \item[(ii)] $n=mr$;
  \item[(iii)] $n=p^emr$, for some $e\geq 1$.
\end{itemize}
\end{theorem}

 In the notation of Theorem \ref{mst}, if $(\Phi^m)'(P)=0$ modulo $\pid$, then we set $r=\infty$. Thus, if $P$ is a periodic point, then the cases (ii) and (iii) are not possible with $r=\infty$.

 \begin{proposition}\label{5.2}\emph{\cite[Proposition 5.2]{MS}}
Let $\phi\colon \p_1\to\p_1$ be a morphism defined over $K$ with good reduction at a place $\pid$. Then for any $P,Q\in\p(K)$ we have
\bdm  \delta_{\mathfrak{p}}(\phi(P),\phi(Q))\geq \delta_{\mathfrak{p}}(P,Q).\edm
\end{proposition}

\begin{lemma}\label{pab}
Let
\begin{equation}\label{cn} P=P_{-m+1}\mapsto P_{-m+2}\mapsto \ldots\mapsto P_{-1}\mapsto P_0=[0:1]\mapsto [0:1].\end{equation}
 be an orbit for an endomorphism $\phi$ defined over $K$ with good reduction outside $S$. For any $a,b$ integers such that $0<a<  b\leq m-1$ and $\pid\notin S$,  it holds
\begin{itemize}
\item[a)] $\dpid(P_{-b},P_0)\leq \dpid(P_{-a},P_0);$
\item[b)] $\dpid(P_{-b},P_{-a})= \dpid(P_{-b},P_0)$.
\end{itemize}
\end{lemma}
\begin{proof}
a) It follows directly from Proposition \ref{5.2}.

b)  By Proposition \ref{5.1} and part a) we have
$$\dpid(P_{-b},P_{-a})\geq \min\{\delta_{\mathfrak{p}}(P_{-b},P_{0}),\delta_{\mathfrak{p}}(P_{-a},P_{0})\}=\dpid(P_{-b},P_{0}).$$
Let $r$ be the largest positive integer such that $-b+r(b-a)<0$. Then
\begin{align*}\delta_{\mathfrak{p}}(P_{-b},P_{0})&\geq \min\{\delta_{\mathfrak{p}}(P_{-b},P_{-a}),\delta_{\mathfrak{p}}(P_{-a},P_{b-2a}),\ldots ,\delta_{\mathfrak{p}}(P_{-b+r(b-a)},P_0)\}\\\nonumber&=\delta_{\mathfrak{p}}(P_{-b},P_{-a}).\end{align*}
The inequality is obtained by applying Proposition \ref{5.1} several times. 
\end{proof}

\begin{lemma}[Lemma 3.2 \cite{CP}] \label{=p} Let $K$ be a function field of degree $D$ over $\F_p(t)$ and $S$ a non empty finite set of places of $K$. Let $P_i\in \p_1(K)$ with $i\in\{0,\ldots n-1\}$ be $n$ distinct points such that

\beq\label{=}\dpid(P_0,P_1)=\dpid(P_i,P_j),\ \ \text{for each distinct $0\leq i,j\leq n-1$ and for each $\pid\notin S$.}\eeq
Then $n\leq (p|S|)^{2D}$.
\end{lemma}

Since $\F_p(t)$ is a principal ideal domain, every point in $\p_1(\F_p(t))$ can be written in $S$--coprime coordinates, i. e., for each  $P\in \p_1(\F_p(t))$ we may write $P=[a:b]$ with $a,b\in R_S$ and $\min\{v_\pid(a), v_\pid(b)\}=0$, for each $\pid\notin S$. We say that $[a:b]$ are $S$--coprime coordinates for $P$.

\bigskip\noindent \textbf{Proof of Theorem \ref{cffzs}}  We use the same notation of Theorem \ref{mst}.  Assume that $S$ contains  only the place at infinity.
\emph{Case $p=2$.}
Let $P\in \p_1(\F_p(t))$ be a periodic point for $\phi$.  Without loss of generality we can suppose that $P=[0:1]$. Observe that $m$ is bounded by 3 and $r=1$. By Theorem \ref{mst}, we have $n=m\cdot 2^e$, with $e$ a non negative integral number. Up to considering the $m$--th iterate of $\phi$, we may assume that the minimal periodicity of $P$ is $2^e$. So now suppose that $n=2^e$, with $e\geq 2$. Consider the following 4 points of the cycle:
$$[0:1]\mapsto [x_1:y_1]\mapsto [x_2:y_2]\mapsto [x_3:y_3]\ldots$$
where the points $[x_i:y_i]$ are written $S$--coprime integral coordinates for all $i\in \{1,\ldots,n-1\}$. By applying Proposition \ref{6.1}, we have $\dpid([0:1],P_1)=\dpid([0:1],P_3)$,  i. e. $x_3=x_1$, because of  $R_S^*=\{1\}$. From $\dpid([0:1],P_1)=\dpid(P_1,P_2)$ we deduce
\begin{equation} \label{eq1} y_2=\frac{x_2}{x_1}y_1+1. \end{equation}

\noindent Furthermore, by Proposition \ref{6.1} we have $\dpid([0:1],P_1)=\dpid(P_2,P_3)$. Since  $x_3=x_1$, then

\begin{equation} \label{eq2} y_3x_2-x_3y_2=x_1. \end{equation}

\noindent This last equality combined with \eqref{eq1} provides $y_3=y_1$, implying $[x_1:y_1]=[x_3:y_3]$.
Thus $e\leq 1$ and $n\in\{1,2,3,6\}$. The next step is to prove that $n\neq 6$. If $n=6$, with few calculations one sees that the cycle has the following form.
\begin{equation}\label{c6} [0:1]\mapsto [x_1:y_1]\mapsto [A_2x_1:y_2]\mapsto [A_3x_1:y_3]\mapsto [A_2x_1:y_4]\mapsto
[x_1:y_5]\mapsto [0:1],\end{equation}
\noindent where $A_2,A_3\in R_S$ and everything is written in $S$-coprime integral coordinates. We may apply Proposition \ref{6.1}, then by considering the $\pid$--adic distances $\delta_\pid(P_1,P_i)$ for all indexes $2\leq i\leq 5$ for every place $\pid$, we obtain that there exists some $S$--units $u_i$ such that
\begin{equation} \label{beginning4}y_2=A_2y_1+u_2;\ \ y_3=A_3y_1+A_2u_3;\ \ y_4=A_2y_1+A_3u_4;\ \ y_5=y_1+A_2u_5.\end{equation}
 Since $R_S^*=\{1\}$, we have that the identities in (\ref{beginning4}) become
$$y_2=A_2y_1+1;\ \ y_3=A_3y_1+A_2;\ \ y_4=A_2y_1+A_3;\ \ y_5=y_1+A_2$$
where $A_2,A_3$ are non zero elements in $\F_p[t]$.
By considering the $\pid$--adic distance $\delta_\pid(P_2,P_4)$ for each finite place $\pid$, from Proposition \ref{6.1} we obtain that
$$v_\pid(A_2x_1)=\delta_\pid(P_2,P_4)=v_\pid(A_2x_1(A_2y_1+A_3)-A_2x_1(A_2y_1+1))=v_\pid(A_2A_3x_1-A_2x_1),$$
i. e. $A_2x_1=A_2A_3x_1-A_2x_1$ (because $R_S^*=\{1\}$). Then $A_2A_3x_1=0$ that contradicts $n=6$. Thus $n\leq 3$.

Suppose now that $P$ is a preperiodic point. Without loss of generalities we can assume that the orbit of $P$ has the following shape:

\begin{equation}\label{cn} P=P_{-m+1}\mapsto P_{-m+2}\mapsto \ldots\mapsto P_{-1}\mapsto P_0=[0:1]\mapsto [0:1].\end{equation}
Indeed it is sufficient to take in consideration a suitable iterate $\phi^n$ (with $n\geq 3$), so that the  orbit of the point $P$, with respect the iterate $\phi^n$, contains a fixed point $P_0$. By a suitable conjugation by an element of  PGL$_2(R_S)$,  we may assume that $P_0=[0:1]$.

For all $1\leq j\leq m-1$, let $P_{-j}=[x_j:y_j]$  be  written in $S$--coprime integral coordinates. By Lemma \ref{pab}, for every $1\leq i<j\leq m-1$ there exists $T_{i,j}\in \srs$ such that $x_i=T_{i,j}x_j$. Consider the $\pid$--adic distance between the points $P_{-1}$ and $P_{-j}$. Again by Lemma \ref{pab}, we have
$$\dpid(P_{-1},P_{-j})=v_\pid(x_1y_j-x_1y_1/T_{1,j})=v_\pid(x_1/T_{1,j}),$$
for all $\pid\notin S$. Then, there exists $u_j\in\srsu$ such that
$y_j=\left(y_1+u_j\right)/{T_{1,j}}$,
for all $\pid\notin S$. Thus, there exists $u_j\in\srsu$ such that $[x_{-j},y_{-j}]=[x_1,y_1+u_j]$.
Since $R_S^*=\{1\}$, then $P_{-j}=[x_1:y_1+1]$. So the length of the orbit \eqref{cn} is at most 3.
We get the bound 9 for the cardinality of the orbit of $P$.

\emph{Case $p>2$.}

Since $D=1$ and $|S|=1$, then the bound  for the number of consecutive points as in Lemma \ref{=p} can be chosen equal to $p^{2}$. By
 Theorem \ref{mst} the minimal periodicity $n$ for a periodic point $P\in \p_1(\Q)$ for the map $\phi$ is of the form $n=mrp^e$ where $m\leq p+1$, $r\leq p-1$ and $e$ is a non negative integer.

 Let us assume that $e\geq 2$. Since $p>2$,  by Proposition \ref{6.1}, for every $k\in\{0,1,2,\ldots, p^{e-2}\}$ and $i\in\{2,\ldots,p-1\}$,  we have that $\dpid(P_0,P_1)=\dpid(P_0,P_{k\cdot p+i})$, for any $\pid\notin S$. Then $P_{k\cdot p+i}= [x_1,y_{k\cdot p+i}]$. Furthermore $\dpid(P_0,P_1)=\dpid(P_0,P_{k\cdot p+i})$  implying that there exists a element $u_{k\cdot p+i}\in R_S^*$ such that
\beq\label{k}P_{k\cdot p+i}=[x_1:y_1+u_{k\cdot p+i}].\eeq

Since $R_S^*$ has $p-1$ elements and there are $(p^{e-2}+1)(p-2)$ numbers of the shape $k\cdot p+i$ as above, we have  $(p^{e-2}+1)(p-2)\leq p-1$. Thus $e=2$ and $p=3$.

 Then $n\leq 72$ if $p=3$ and $n\leq (p^2-1)p$ if $p\geq 5$.  For the more general case when $P$ is preperiodic, consider the same arguments used in the case when $p=2$, showing $[x_{-j},y_{-j}]=[x_1,y_1+u_j]$, with $u_j\in\srsu$.
Thus,  the orbit of a point  $P\in \p_1(\Q)$ containing $P_0\in \p_1(\Q)$, as in  (\ref{cn}),  has length at most $|R_S^*|+2=p+1$. The bound in the preperiodic case is then $288$ for $p=3$ and $(p+1)(p^2-1)p$ for $p\geq 5$. \hspace{0.5cm} $\Box$

\medskip
 With similar proofs, we can get analogous bounds for every finite extension $K$ of $\F_p(t)$.  The bounds
of Theorem \ref{cffzs}, with $K=\F_p(t)$, are especially interesting, for they are very small.

\end{document}